\documentclass[a4paper,11pt]{amsart}
\usepackage{amscd,amsmath,amsthm,amssymb,marvosym}
\usepackage[english]{babel}
\usepackage[all]{xy}
\usepackage[colorlinks=true, linkcolor=blue, citecolor=blue]{hyperref}
\usepackage{etoolbox}
\usepackage{cite}

\begin{document}

	\newtheorem{theorem}{Theorem}[section]
	\newtheorem{lemma}[theorem]{Lemma}
	\newtheorem{proposition}[theorem]{Proposition}
	\newtheorem{corollary}[theorem]{Corollary}
	\newtheorem{question}[theorem]{Question}
	\newtheorem{conjecture}[theorem]{Conjecture}
	
	\theoremstyle{definition}
	\newtheorem{definition}[theorem]{Definition}
	\newtheorem{example}[theorem]{Example}

	\theoremstyle{remark}
	\newtheorem{remark}[theorem]{Remark}

	\newcommand{\K}{\mathbb{K}}
	\newcommand{\R}{\mathbb{R}}
	\newcommand{\C}{\mathbb{C}}
	\newcommand{\Q}{\mathbb{Q}}
	\newcommand{\Z}{\mathbb{Z}}
	\newcommand{\Proj}{\mathbb{P}}
	\newcommand{\A}{\mathbb{A}}
	\newcommand\blfootnote[1]{%
		\begingroup
		\renewcommand\thefootnote{}\footnote{#1}%
		\addtocounter{footnote}{-1}%
		\endgroup
	}

	\newcommand{\NS}{NS}
	\address{Fabrizio Anella\\ Dipartimento di Matematica e Fisica\\ Universit\`{a} Roma 3\\ Largo San Leonardo Murialdo 1\\ 00146\\ Rome\\ Italy}
	\email{fabrizio.anella2@uniroma3.it}

		\author{Fabrizio Anella}
		\date{March 2019}
		\title{Rational curves on genus one fibrations}
		
	\begin{abstract}
		In this paper we look for necessary and sufficient conditions for a genus one fibration to have rational curves.
		%These conditions depends on the singularities that we allows.
		%Without asking anything on the base of the fibration, we can say something only on the rational curves that are vertical for the fibration. %We prove that a projective variety with log terminal singularities that admits a genus one fibration contains vertical rational curves if an only if it is an \'etale quotient of 
		We show that a projective variety with log terminal singularities that admits a relatively minimal genus one fibration $X\rightarrow B$ does contain vertical rational curves if and only if it not isomorphic to a finite \'etale quotient of a product $\tilde{B}\times E$ over $B$. Many sufficient conditions for the existence of rational curves in a variety that admits a genus one fibration are proved in this paper.
	\end{abstract}
\maketitle
\blfootnote{\textup{2010} \textit{Mathematics Subject Classification}: Primary: 14J32; Secondary: 14E30, 14D06.} 
\blfootnote{\emph{Key words and phrases.} Elliptic fiber space, Calabi--Yau variety, genus one fibration, rational curve, augmented irregularity.}

	\section*{Introduction}
	%This article can be seen as a natural continuation of another paper of the author.
The starting point of this work has been the following folklore conjecture.

\begin{conjecture}
	Every (possibly singular) Calabi--Yau variety does contain rational curves.
\end{conjecture}

This conjecture is unsolved even for smooth Calabi--Yau manifolds in dimension three. We started studying Calabi--Yau varieties that admit a genus one fibration and we got a positive answer in \cite{Anella}. At this point it is natural to ask the following question.

\begin{question}\label{question}
	Under what conditions a genus one fibration does contain rational curves?
\end{question}
Without asking anything on the base of the fibration, we can say something only on the rational curves that are vertical for the fibration. 
The main purpose of this article is to prove the following answer to Question \ref{question} that gives a complete characterization in the case of relatively minimal genus one fibration between $\mathbb{Q}$-factorial varieties.

\begin{theorem}\label{ifandonlyifintro}
	Let $X\xrightarrow{\pi} B$ be genus one fibration such that $K_X\equiv_{\textrm{num}}\pi^*L$ for some $\Q$-Cartier $\Q$-divisor on $B$. Suppose moreover that $X$ and $B$ are $\mathbb{Q}$-factorial. Then $X$ does not contain vertical rational curves if and only if there exists a finite cover $\tilde{B} $ of $B$ and a genus one curve $E$ such that there is a finite globally \'etale cover of $X$ isomorphic to $\tilde{B}\times E$ over $B$.
\end{theorem}

A key ingredient for the proof of this theorem is the proof of the following theorem, that is a generalization of the main result in \cite{Anella} and of the main theorem of \cite{diverio2016rational}.
\begin{theorem}\label{maintheoremintro}
	Let $(X,\Delta )$ be a klt pair such that there exists a surjective morphism $\phi:X\rightarrow B$ to a variety of dimension $n-1$. Suppose moreover $K_X+\Delta \equiv_{\textrm{num}} \phi^*L$ for some $\Q$-Cartier $\Q $-divisor $L$ on $B$ and the augmented irregularity of $X$ is zero, then there exists a subvariety of $X$ of dimension $n-1$ covered by rational curves contracted by $\phi$.
\end{theorem}
%This theorem is interesting also in the case there is no boundary, but the experience with the minimal model program
 
This theorem has some interesting consequences like the following. 

\begin{corollary}\label{corollariofigo}
	Let $X$ be a projective variety of dimension $n$ with at most log terminal singularities, $\kappa(X)=n-1$ and $\tilde{q}(X)=0$. Then $X$ does contain rational curves.
\end{corollary}

Using a very nice result of \cite{lazic2018maps} we can do slightly better in the smooth case.
\begin{corollary}\label{lazic}
	Let $X$ be a smooth projective variety of dimension $n\geq 2$ and $\tilde{q}(X)=0$. If $X$ is covered by genus one curves, then it contains a rational curve.
\end{corollary}

In Section \ref{preliminari} we fix some notations and definitions used in the subsequent parts.
In Section \ref{irr} we discuss some properties of the augmented irregularity.
In Section \ref{dim} we prove Theorem \ref{maintheoremintro}. The proof is organized in several lemmas which are of independent interest.
In Section \ref{consegueza} we explain some generalizations and consequences of Theorem \ref{maintheoremintro} and we give other partial answers to Question \ref{question}.
An interesting generalization of Theorem \ref{maintheoremintro} is Theorem \ref{generalmaintheorem}.
There will be a more accurate analysis of Question \ref{question} in the forthcoming PhD thesis of the author.

\subsubsection*{Acknowledgements} The author warmly thanks Raffaele Carbone for the continuous helpful discussion during the writing of this paper and Andreas H\"{o}ring for many useful comments. The author would also like to thank his advisor, Simone Diverio, for the continuous help provided during this work. 
	
	\section{Preliminaries}
	\label{preliminari}

In this paper every variety will be an irreducible projective variety over the complex numbers. The variety $X$ will be always normal and of dimension $n\geq 2$. The notations and standard properties about singularities that is used in this article can be found for example in \cite{kollar2008birational}. For the reader's convenience we recall some definitions.

\begin{definition}
Let $\phi :X\rightarrow B$ be a surjective morphism of normal projective varieties an $D\in \operatorname{WDiv}(X) $ be a prime Weil divisor. We say that $D$ is \emph{exceptional} if $\operatorname{cod}_B(\phi(D))\geq 2$.
We say that $D$ is of \emph{insufficient fiber type} if $\operatorname{cod}_B(\phi(D))=1$ and there exists another prime Weil divisor $D'\neq D$ such that $\phi (D')=\phi(D)$. In either of the above cases, we say that $D$ is \emph{degenerate}.
\end{definition}
\begin{definition}
A morphism $f : Z \rightarrow Y$ between normal varieties is called \emph{quasi-\'etale} if $f$ is quasi-finite and \'etale in codimension one. The \emph{augmented irregularity} of a variety $Y$ is the following, not necessarily finite, positive integer $$\tilde{q}(Y):=\sup\{q(Z)\ | \ Z\rightarrow Y \text{ is a finite quasi-\'etale cover} \}.$$ %Let us introduce some definitions.
\end{definition}
 \begin{remark}
 	The above definition of quasi-\'etale morphism is not the same of \cite{catanese2007}.
 \end{remark}
\begin{definition}
	Let $f:X\rightarrow Y$ a surjective projective morphism of normal variety. The \emph{singular values} of $f$ is the following subset of $Y$
	$$\operatorname{Sv}(f)=\{y\in Y |\ \operatorname{dim}(f^{-1}(y))>\operatorname{dim}(X)-\operatorname{dim}(Y) \vee f^{-1}(y)\text{ is singular} \}.$$
\end{definition}
\begin{remark}
	The singular values of $f$ is the image of the singular locus of $f$. For the interested reader the definition of singular locus of a morphism can be found at the following link: \texttt{\url{http://stacks.math.columbia.edu/tag/01V5}}. We do not give the definition of singular locus of a morphism because the definition is too technical and we just need the given characterization of the image of the singular locus.
\end{remark}
One can associate to any elliptic curve a complex number called its $j$-invariant. This association is modular, which means that a genus one fibration $f:Y\rightarrow B$ comes with a rational map $j:B \dashrightarrow \C $ called $j$-function that is at least defined over the smooth values of $f$.
For some standard facts about the $j$-function of an elliptic family the reference can be found in \cite{kodaira1963compact} or \cite{hartshorne2009deformation}. A brief explanation in the case of genus one fibration can be found in \cite{Anella}.
\begin{definition} \label{CY}
	A variety with at most klt singularities $Y$ (resp. a klt pair $(Y, \Delta)$) is a \emph{Calabi--Yau} variety (resp. a \emph{log Calabi--Yau}) if $K_Y\equiv_{\text{num}}0 $ (resp. $K_Y+\Delta\equiv_{\text{num}}0$) and $\tilde{q}(Y)=0$.
\end{definition}
In some recent works on Beauville--Bogomolov decomposition (see for example \cite{greb2011singular}, \cite{horing2018algebraic}, \cite{druel2018decomposition}) several definitions of Calabi--Yau varieties appeared. In our Definition \ref{CY} we include also products of Calabi--Yau and irreducible holomorphic symplectic varieties in the sense of \cite{greb2011singular}.
\begin{definition}

A variety with at most klt singularities $Y$ (resp. a pair $(Y, \Delta)$) together with a fibration $Y\overset{f}{\rightarrow} B$ is a \emph{Calabi--Yau fiber space} (resp. a \emph{log Calabi--Yau fiber space}) if the generic fiber $Y_t$ of $f$ (resp. the pairs $(Y_t, \Delta_t)$) has numerically trivial canonical bundle (resp. $K_{Y_t}+\Delta_t \equiv_{\text{num}}0$). 
	
\end{definition}
If the general fiber is a curve we are mainly interested in the case where the boundary $\Delta $ does not intersect the general fiber. Indeed if the intersection is non-trivial then $Y$ is uniruled.
\begin{definition}
	An \emph{genus one fibration} is a morphism $\pi:Y\rightarrow B$ between normal projective varieties with connected fibers and such that the general fiber is a smooth genus one curve. 
\end{definition}

A particular case of Calabi--Yau fiber space is the following.

\begin{definition}
	A Calabi--Yau fiber space $Y\overset{f}{\rightarrow} B$ (resp. a log Calabi--Yau fiber space $(Y, \Delta)\overset{f}{\rightarrow} B$) is called a \emph{relatively minimal} (resp. \emph{log) Calabi--Yau fiber space} if $K_Y\equiv_{\text{num}}f^* L$ ($K_Y+\Delta\equiv_{\text{num}}f^* L$) for some $\Q$-Cartier $\Q $-divisor on $B$.
\end{definition}
In particular when we say that genus one fibration is relatively minimal we are saying that $X$ has at most log terminal singularities.

An important class of examples of Calabi--Yau fiber spaces is given by orbibundles. We just recall the construction of orbibundles because they are a key point in the proof of Theorem \ref{maintheoremintro}. More properties and details can be found in the article of Koll\'ar \cite{kollar2015deformations}.

Let $\tilde{B}$ be a normal variety, $F$ a variety with at most klt singularities, $K_F \equiv_{\text{num}}0 $ and $\tilde{Y}:= \tilde{B}\times F$ their product. Let $G$ be a finite group and $\rho_B:G\rightarrow \operatorname{Aut}(\tilde{B}), $ $\rho_F:G\rightarrow \operatorname{Aut}(F),$ two faithful representations.
\begin{definition}
	An \emph{orbibundle} is the Calabi--Yau fiber space $$(Y\rightarrow B):=\tilde{Y}/ G \rightarrow \tilde{B}/G$$
	obtaind as quotient respect the diagonal representation of $G$.
\end{definition}

	\section{Some remarks on the augmented irregularity}\label{irr}
	It is not difficult to show that the augmented irregularity is a birational invariant for smooth projective varieties. To show this fact we take two smooth birational projective varieties $X$ and $X'$. We can suppose, considering a resolution of the birational map, that there is a well-defined morphism $X'\rightarrow X$. Any quasi-\'etale cover $Z$ of $X$ (hence globally \'etale) can be pulled back to an \'etale cover $Z'$ of $X'$. These two covers are smooth and birational, so $q(Z)=q(Z')$. Since this argument works for any quasi-\'etale cover taking the $\sup$ we get $\tilde{q}(X)=\tilde{q}(X')$. However the augmented irregularity is not a birational invariant for projective varieties with canonical singularities. Indeed the standard construction of a Kummer surface is a counterexample.

\begin{example}
 We take an elliptic curve $E$ and then consider the quotient $X:=(E\times E)/ \pm$ of the product of two copies of $E$ by the involution. The quotient map $E\times E \rightarrow X$ is a quasi-\'etale cover so $\tilde{q}(X)\geq 2$, moreover by \cite[Remark 4.3]{druel2018decomposition} it holds the equality. However a minimal resolution $\tilde{X} $ of $X$ is a K3 surface. In particular $\tilde{X}$ is simply connected, hence $\tilde{q}(\tilde{X})=0$. The variety $X$ is thus also an example of a regular variety with non trivial augmented irregularity.
 
\end{example}
However it follows from the last part of the proof of Theorem \ref{maintheoremintro} that the augmented irregularity is invariant for birational morphisms that are isomorphism in codimension one of projective varieties with at most klt singularities.

It is natural to ask whether there exists some manageable conditions for the vanishing of the augmented irregularity of a variety. It is easy to check \cite[Remark 1.17]{Anella} that a variety $X$, say smooth for simplicity, with finite fundamental group has $\tilde{q}(X)=0$. For varieties with numerically trivial canonical divisor an interesting characterization is given in \cite[Theorem 11.1]{greb2017klt}, where the authors proved that, in this setting, $\tilde{q}(X)=0$ if and only if for any $k>0$ there are no non-trivial symmetric reflexive forms, \textsl{i.e.} $H^0(X,\operatorname{Sym}^{[k]}\Omega^1_X)=0 $ $\forall k>0$. We prove that an implication still holds without the assumption on the canonical bundle: the following is a sufficient condition for the vanishing of the augmented irregularity which does not rely on computations of invariants on quasi-\'etale covers, but only on invariants of the variety under investigation.

\begin{proposition}
	Let $X$ be a projective variety with at most log terminal singularities. If $H^0(X, \operatorname{Sym}^{[k]}\Omega^1_X)=0$ for every $k>0$, then $\tilde{q}(X)=0$.
\end{proposition}
Let us recall that by $\operatorname{Sym}^{[k]}\Omega_X^1$ we mean the sheaf of the reflexive symmetric forms on $X$.
\begin{proof}
	Suppose by contradiction that there is a quasi-\'etale cover $\tilde{X} \rightarrow X$ with $H^1(\tilde{X},\mathcal{O}_{\tilde{X}})\neq 0$. The variety $\tilde{X}$ is klt \cite[Proposition 5.20]{kollar2008birational} and by \cite[Proposition 6.9]{greb2017klt} there is a non-zero reflexive form $\omega \in H^0 (\tilde{X}, \Omega^{[1]}_{\tilde{X}})$. By definition the sections of a reflexive sheaf are exactly the sections on the regular part of $X$. So to construct a non-zero global section of $\operatorname{Sym}^{[k]}\Omega^1_X$ we construct an element in $H^0(X_{\text{reg}}, \operatorname{Sym}^{k}\Omega^1_{X_{\text{reg}}})$. Now we consider just the restriction to the regular locus of $X$:
	
	$$Y:=\tilde{X}\times_X X_{\text{reg}}\rightarrow X_{\text{reg}}.$$
	This is a finite \'etale cover, so we can find a further \'etale cover, finite over the regular part $\tilde{Y}\rightarrow X_{\text{reg}}$, that is Galois. Let $G$ be the group of deck transformations of $\tilde{Y}$ over $X_{\text{reg}}$. By abuse of notations we call again $\omega$ the pullback to $\tilde{Y}$ of $\omega$. Now consider the section $ \tilde{\alpha}:= \sum_{\tau \in G} \bigotimes_{\rho \in G} \tau^* \rho^* \omega \in H^0(\tilde{Y},  (\Omega^1_{\tilde{Y}} )^{\otimes N})$ where $N=|G|$. This section is invariant under the action of the deck trasformations, so it descends to a section $\alpha$ of $ H^0(X_{\text{reg}},  (\Omega^1_{X_{\text{reg}}} )^{\otimes N}).$ By construction it is easy to check that this section is symmetric, \textsl{i.e.} $\omega $ belongs to $H^0(X_{\text{reg}}, \operatorname{Sym}^N (\Omega^1_{X_{\text{reg}}} )).$ It is less trivial to prove that $\tilde{\alpha}$, and hence $\alpha$, is non-zero. 
	
	For any non-zero element $\gamma \in H^0(\tilde{Y},  \Omega^1_{\tilde{Y}})$ and a generic point $p\in \tilde{Y}$ the space $\operatorname{Ker}(\gamma)\subset T_{\tilde{Y}, p}$ is a proper subspace. Since $\omega\neq 0$ and the elements $\rho \in G$ are automorphisms (and we are working over $\C$), also $\rho^* \omega$ are non-zero elements in $H^0(\tilde{Y},  \Omega^1_{\tilde{Y}})$. So for generic $p \in \tilde{Y}$ we can choose a tangent vector $$0\neq v\in T_{\tilde{Y}, p}\setminus \bigcup_{\rho \in G} \operatorname{Ker}((\rho^* \omega)_p).$$ Now we can evaluate our section $\tilde{\alpha}$ at the vector $v^{\otimes N}$. The computations are the following:
	$$\tilde{\alpha}(v^{\otimes N})=\sum_{\tau \in G} \bigotimes_{\rho \in G} \tau^* \rho^* \omega(v)=\sum_{\tau \in G}\prod_{\rho \in G} \tau^* \rho ^* \omega (v)=N\prod_{\rho \in G} \rho ^* \omega (v)\neq 0.$$
	So we have constructed a non-zero section of $ H^0(X_{\text{reg}},  \operatorname{Sym}^N \Omega^1_{X_{\text{reg}}} ))$ that corresponds to a non-zero section of  $ H^0(X,  \operatorname{Sym}^{[N]} \Omega^1_X ).$
	% It remains just to prove that this section is symmetric. The point is that the deck trasformations acts trivially on the vector bundles that come from $X_{\text{reg}}$. The section $\omega\otimes \rho_1^*\omega\otimes\dots \otimes\rho_N^*\omega$ in not symmetric but it descends to a symmetric element in $x_{\text{reg}}$. To prove that it is symmetric we can work locally. Let $v_1, \dots , v_N$ be vector fields over an open subset $U\subset X_{\text{reg}}$ and $v_{i_1}, \dots , v_{i_N}$ any permutation. We should prove that $\alpha(v_1\otimes\dots \otimes v_N)=\alpha(v_{i_1}\otimes\dots \otimes v_{i_N})$. The computations are the following
	%$$\alpha(v_1\otimes\dots \otimes v_N)=\omega\otimes \rho_1^*\omega\otimes\dots \otimes\rho_N^*\omega(v_1\otimes\dots \otimes v_N)=$$
	%$$=\omega(f^* v_1)\cdot \omega(\rho_{2*} f^* v_2)\cdots \omega(\rho_{N*} f^* v_N))=\omega(f^* v_1)\cdot \omega(f^* v_2)\cdots \omega(f^* v_N))=$$
	%$$=\omega(f^* v_{i_1})\cdot \omega(\rho_{2*} f^* v_{i_2})\cdots \omega(\rho_{N *} f^* v_{i_N}))=\alpha(v_{i_1}\otimes\dots \otimes v_{i_N})$$
	\end{proof}
	\section{Proof of the main theorem}
	\label{dim} 
	We start with a lemma that is already stated in \cite{diverio2016rational}.
\begin{lemma}
	\label{isocod}
	Let $X\overset{\pi}{\rightarrow} B$ be a genus one fibration. If the subvariety of the singular values $\operatorname{Sv}(\pi)\subset B$  has codimension at least two in $B$ then the family $\pi $ is isotrivial.
\end{lemma}
For the proof of this result we refer to \cite[Lemma 1.15]{Anella}

%\begin{remark}\label{iso2}
%	Since $B$ is normal the condition $cod_B(Z)\geq 2$ is equivalent to $cod_B(\pi(sing(\pi))\geq 2$. Since the locus where the preimage of a point has codimension greater than two has codimension at least two in $B$ this is also equivalent to $Z_0:=cod_B(\pi(sing(\pi))\cap \pi(Exc(\pi)))\geq 2$.
%\end{remark}
	Now we study the general fibers over the singular values of an genus one fibration.

\begin{lemma}\label{singularfiber}
	Let $\phi:X\rightarrow B$ be a genus one fibration and $Z:=\operatorname{Sv}(\phi)$. Suppose that $cod_B(Z)=1$, then a general fiber over $Z$ is $ m E+\sum m_i R_i$ where $E$ is a genus one curve and $R_i$ are rational curves.
\end{lemma}
\begin{proof}
	We can study the restriction of $\phi$ to a surface as follows.
	Let $H$ be a very ample divisor on $B$ such that $(n-2)H+L$ is globally generated. The pullback $\phi^*H$ is a globally generated Cartier divisor. Moreover there is an isomorphism 
	$$H^0(X,\phi^*H)\simeq H^0(B,\phi_*(\phi^*H))\simeq H^0(B,H)$$
	because $\phi $ has connected fibers. This implies that general elements in $|H|$ are general also in $|\phi^*(H)|$. So we choose $n-2$ general divisors $D_1, \ldots , D_{n-2} \in |H|$ such that $C:=D_1\cap \ldots \cap D_{n-2}$ is a smooth irreducible curve in $B_{\textrm{reg}}$ not contained in the locus of singular values of $\phi$ and $S:=\phi^{-1}(D_1)\cap \ldots \cap \phi^{-1}(D_{n-2})$ is a normal surface. Looking at the Kodaira's table \cite[Section V.7]{barth2015compact} it is easy to check that the singular fibers of $\phi|_S$ are $ m  E+\sum m_i R_i$ where $E$ is an elliptic curve and $R_i$ are rational curves. The condition on the dimension of $Z$ insures that a general point in $Z$ lies on a curve obtained as general intersection of hyperplane sections.
\end{proof}
\begin{remark}\label{singellittiche}
	If $X$ contains no uniruled codimension one subvarieties but $\operatorname{Sv}(\phi)$ has codimension one in $B$, then the fibers over any general point of $\operatorname{Sv}(\phi)$ of dimension $n-2$ is a multiple genus one curve.
\end{remark}

\begin{remark}\label{remarkinsufficienttype}
	It follows from Lemma \ref{singularfiber} and from a result of Kawamata \cite{kawamata1991length} that a minimal genus one fibration (with at most log terminal singularities) with no uniruled codimension one subvarieties has no degenerate divisors.
\end{remark}

 Lemma \ref{singularfiber} can be seen as a soft version of Kodaira's table in higher dimension. With the same strategy of the proof of this lemma one can certainly do a better classification of singular fibers. Using the techniques of Lemma \ref{singularfiber} we can control only the general singular fiber in codimension one. Other fibers may appear in greater codimension.
	Now we can merge together these lemmas and prove the following result.

\begin{lemma}\label{isotri}
	Let $X\overset{\phi}{\rightarrow} B$ be a genus one fibration between normal varieties such that $X$ does not contain codimension one subvarieties that are uniruled. Then the family $X\overset{\phi}{\rightarrow} B$ is isotrivial.
\end{lemma}
This lemma should be compared with \cite[Proposition 6.5]{lazic2018maps} and \cite[Proof of Corollary 3.34]{voisin2003some}.
\begin{proof}[Proof of Lemma \ref{isotri}]
	We can assume by Lemma \ref{isocod} that $Z:=\operatorname{Sv}(\phi))$ has codimension one in $B$. The general fiber over $Z$ is classified by Lemma \ref{singularfiber}. Since there are no uniruled codimension one subvarieties, in the general fiber over $Z$ there are only multiple genus one curves.
	
	Now we can proceed cutting with hyperplane sections as in the proof of Lemma \ref{singularfiber}. In this way we get many curves $C$ in $B$ with only genus one fibers (possibly multiple) over them. Up to consider a finite possibly ramified base change we can assume this map has a section. 
	The $j$-invariant for multiple elliptic curves is well-defined as one can easily check with a semistable reduction. Since the curve $C$ is complete this implies that the $j$-map is constant, \textsl{i.e.} the family $\pi$ restricted over $C$ is isotrivial.
	
	For each curve $C$ obtained in this way we get an isotrivial family. From this fact it follows that the all the family over $B$ is isotrivial. Let us prove this fact by induction on the dimension of $B$. 
	
	There is nothing to prove if the dimension is one. By induction we can suppose that the family is isotrivial when restricted to an ample subvariety $H\subset B$. The fibers over curves that are general complete intersections are pairwise isomorphic and these curves must intersect $H$. The union of these curves dominates $B$. This implies that the family $ X\overset{\phi}{\rightarrow} B$ is isotrivial.
\end{proof}
Another way to prove this lemma is to consider the $j$-map directly from $B$. 

	Finally we can proceed with the proof of Theorem \ref{maintheoremintro}. For the reader's convenience we state again the result that we are going to prove.
\begin{theorem}
	Let $(X,\Delta )\xrightarrow{\phi} B$ be a relatively minimal log Calabi--Yau fiber space such that $\tilde{q}(X)=0$ and $\operatorname{dim}(B)=n-1$. Then there exists a subvariety of $X$ of dimension $n-1$ covered by rational curves contracted by $\phi$.
\end{theorem}

\begin{proof}
	We can suppose, eventually passing to the Stein factorization, that $\phi$ has connected fibers and $B$ is normal.
	
%	Since $K_X+\Delta$ is $\phi $-trivial and $(X,\Delta)$ is klt, by a result of Kawamata \cite{kawamata1991length} the $\phi$-exceptional locus is uniruled. So we can suppose that there are no $\phi$-exceptional divisors.
	
	Since $K_X+\Delta \equiv_{\text{num}} \phi^*L$, the restriction of $(K_X+\Delta)$ to a general fiber of $\phi$, \textsl{i.e.} $(K_X+\Delta)|_{X_t}$ is numerically trivial. It follows from standard arguments that a general fiber of $\phi$ is a smooth curve contained in the smooth locus of $X$, so by adjunction formula $K_{X_t}\sim K_X |_{X_t}$. This implies that $K_{X_t}\equiv_{\text{num}} -\Delta|_{X_t}$ and hence the general fiber has genus at most one. If the genus is zero the variety $X$ is uniruled, so we can suppose $\phi $ is a genus one fibration. Note that even if $\Delta$ and $K_X$ are not $\Q$-Cartier $\Q$-divisor their restriction on a neighborhood of a general fiber $X_t$ is $\Q$-Cartier. 
	
	If the fibration is non-isotrivial then there exists a uniruled divisor in $X$ by Lemma \ref{isotri}. It remains to study the case $\phi $ is a genus one fibration without exceptional divisors. Under these conditions by a result of Koll\'ar \cite[Theorem 44]{kollar2015deformations} $X$ is birational over $B$ to an orbibundle $X_{\text{orb}}$. By construction the variety $X_{\text{orb}}$ has no $\phi$-exceptional divisors. Using that $X$ and $X_{\text{orb}}$ are birational over $B$, both have no degenerate divisors (see Remark \ref{remarkinsufficienttype}) and they are relatively minimal over $B$, we know that $X$ and $X'$ are isomorphic in codimension one. Indeed under these conditions the birational morphism can't contract any divisor over $B$. Since $\tilde{X}:=\tilde{B}\times F\rightarrow X_{\text{orb}}$ is quasi-\'etale \cite[see Lemma 38]{kollar2015deformations} there exists an open subset $\tilde{X}^{\circ}\subset \tilde{X}$ such that the induced morphism $\tilde{X}^{\circ}\rightarrow X$ is quasi-finite, \'etale and $\operatorname{cod}_{\tilde{X}}(\tilde{X}^{\circ})\geq 2$. By Zariski Main Theorem a quasi-finite morphism factorizes through an open immersion and a finite morphism, so there exists a commutative diagram
		\[ \xymatrix{
		\tilde{X}^{\circ} \ar[r] \ar[rd] & Y \ar[d]^{\psi} \\
		    & X 	} \]
with $Y$ projective and the morphism $\psi$ is \'etale over the image of $\tilde{X}^{\circ}$. Since $\psi $ is finite and the complementary of the image of $\tilde{X}^{\circ}$ in $X$ has codimension at least two, the morphism $\psi$ is a quasi-\'etale cover of $X$. By some basic properties of the reflexive sheaves and by \cite[Proposition 6.9]{greb2011singular} the following chain of isomorphisms holds
$$\overline{H^1(Y, \mathcal{O}_Y )}\simeq 
H^0(Y, \Omega_Y^{ [ 1 ] } )\simeq H^0(\tilde{B}\times F, \Omega_{\tilde{B}\times F}^{ [ 1 ] } )\simeq $$
$$\simeq H^0(\tilde{B}, \Omega_{\tilde{B}}^{ [ 1 ] } )\oplus H^0( F, \Omega_{ F}^{ [ 1 ] } )\neq 0. $$ 
So $Y$ is a quasi-\'etale cover of $X$ that has positive irregularity and this is a contradiction. \end{proof}
	\section{Applications and remarks}
	\label{consegueza}
	In this section we talk about some applications of Theorem \ref{maintheoremintro}. 
We start noting the following fact that will be useful also in the following.
\begin{remark}
	In Theorem \ref{maintheoremintro} the condition on the augmented irregularity is needed only to avoid the case $X$ is isomorphic in codimension one to an orbibundle. More precisely the hypothesis $\tilde{q}(X)=0$ in Theorem \ref{maintheoremintro} (and also in the following corollaries) can be replaced by the following weaker condition: $X$ is not isomorphic in codimension one over $B$ to a quasi-\'etale quotient of $\tilde{B}\times E$ for some cover $\tilde{B}\rightarrow B$. 	
\end{remark} 
The following theorem is a partial answer to Question \ref{question} in the case there are some exceptional divisors.
\begin{theorem}\label{generalmaintheorem}
	Let $(X,\Delta)$ be a klt pair with $\tilde{q}(X)=0$ such that there exists a surjective morphism $\phi:X\rightarrow B$ to a variety of dimension $n-1$. Suppose moreover $K_X+\Delta \equiv_{\textrm{num}} \phi^*L+\sum a_i E_i$ for some $\Q$-Cartier $\Q $-divisor $L$ on $B$, some $\phi$-exceptional divisor $E_i$ and whose coefficients are not all strictly negative. Then $X$ does contain rational curves. 
\end{theorem}

\begin{proof}
	Let us start supposing that $X$ is $\Q$-factorial. We can write $K_X+\Delta \equiv_{\text{num}} \phi^*L +D-D'$ with $D $ and $D'$ $\phi$-exceptional effective divisors with no common components. If there are no exceptional divisors this is Theorem \ref{maintheoremintro}. Otherwise it follows from Hodge Index Theorem that there exists a component $E$ of the divisor $D$ that is covered by curves that intersect negatively $K_X+\Delta$ \cite[Lemma 2.9]{laivarietiesfibered}. By Cone Theorem this implies that there are rational curves in $X$. 
	
	If $X$ is not $\Q$-factorial, in order to apply the result of Lai, we consider a $\Q$-factorialization $X_{\Q}\rightarrow X$ and with the same argument we get that there exists a divisor in $X_{\Q}$ covered by curves that intersect negatively $K_{X_{\Q}}+\Delta_{\Q}$. Since a $\Q$-factorialization is small, not all these curves can be contracted in $X$. By projection formula we get curves in $X$ with negative intersection with $K_X+\Delta$. The conclusion follows again by the Cone Theorem.  %Now we want to apply Bend and Break Theorem, but we cannot do this directly because these curves are not in general contained in the regular part of $X$.
	%Let $\nu:X_{\text{term}}\rightarrow X$ be a terminalization of $X$. We can write $$K_{X_{\text{term}}}+\Delta' \equiv_{\textrm{num}}\nu^* (K_X+\Delta)-\sum a_i E_i$$ for some non negative number $a_i$. By projection formula the strict transform of the curves that intersect negatively $K_X+\Delta$ has negative intersection with $K_{X_{\text{term}}}+\Delta'$. Since 
\end{proof}
\begin{remark}\label{logcanonicaltre}
	Up to consider the new pair $K_X+\Delta+\varepsilon E_i$ one can prove that there exists rational curves, provided that not every $-a_i$ is bigger than the log-canonical threshold of $E_i$.
\end{remark}
\begin{corollary}
	In the setting of Theorem \ref{generalmaintheorem} with the further condition that $X$ is smooth in codimension two, \textsl{e.g.} $X$ has terminal singularities. Then $X$ contains an uniruled divisor.
\end{corollary}
\begin{proof}
	Since we are looking for a uniruled divisor and a $\Q$-factorialization is small, we can assume that $X$ is $\Q$-factorial.
	By the proof of Theorem \ref{generalmaintheorem} there exists a very ample line bundle $H$ and an irreducible component $E$ of $D$ such that $E\cdot H^{n-2}\cdot (K_X+\Delta)<0$. Since $X$ is smooth in codimension two the general complete intersection of $n-2$ elements in $|H|$ is contained in $X_{\textrm{reg}}$, so the general element in $E\cap |H|\cap \dots \cap |H|$ is a curve in the regular part of $X$ which intersects negatively $K_X+\Delta$. So to conclude it is sufficient to apply Bend and Break Theorem \cite[Theorem 3.6]{Deb01}.
\end{proof}
%\begin{remark}
%	The condition on the augmented irregularity in Theorem \ref{maintheoremintro} and \ref{generalmaintheorem} is necessary to exclude that the orbibundle associated to $X$ is a quasi-\'etale quotient of $\tilde{B}\times E$. Theorem \ref{generalmaintheorem}can be generalized replacing the condition $\tilde{q}(X)=0$ with the condition \lq \lq there is no \'etale map $E\times B^0\overset{\pi}{\rightarrow} X$ such that $\operatorname{cod}_{X}(\pi(E\times B^0)^c)\geq 2$\rq \rq.
%\end{remark}
Returning to Question \ref{question}, we can give a complete answer (for vertical rational curves) in the case of a relatively minimal genus one fibration between $\mathbb{Q}$-factorial varieties.
\begin{theorem}\label{ifandonlyif}
	Let $X\xrightarrow{\psi} B$ relatively minimal genus one fibration such that $X$ and $B$ are $\mathbb{Q}$-factorial. Then $X$ does not contain vertical rational curves if and only if there is a finite globally \'etale cover of $X$ isomorphic to $\tilde{B}\times E$ over $B$, for some cover $\tilde{B}$ of $B$.
\end{theorem}
\begin{proof}
	$(\Leftarrow)$ Suppose there is a finite cover $\tilde{B} \rightarrow B$ such that we have the following diagram
		\[ \xymatrix{
	\tilde{B}\times E\ar[r]^\phi \ar[d]^\pi & X\ar[d]^\psi	\\
\tilde{B}\ar[r]& B} \]
where $\phi $ is finite and globally \'etale. The restriction of $\phi$ to any fiber $\pi^{-1}(t)$ is \'etale, so the image in $X$ is again a genus one curve in $X$. Any fiber of $\psi$ is the image of a curve obtained in this way, so all the vertical curves have genus one.

 $(\Rightarrow)$ Since $\psi$ is relatively minimal the exceptional locus of $\psi$ is covered by rational curves contracted by $\psi$ by \cite{kawamata1991length}. So we can suppose that the $\psi $ is equidimensional. Moreover by Lemma \ref{isotri} we can assume that $\psi$ is isotrivial. Under these conditions we can apply \cite[Theorem 44]{kollar2015deformations} and obtain that $X$ is isomorphic to an orbibundle. This means that there is a finite quasi-\'etale cover $\tilde{B}\times E\xrightarrow{f} X$ over $B$. This map is not globally \'etale by assumption, so $f$ ramifies at some point $(t, z)$. The restriction of $f$ to the curve ${t}\times E$ is a map which ramifies from a genus one curve, so by Riemann--Hurwitz formula the image is a rational curve. By construction this curve is vertical for $\psi$.
\end{proof}
\begin{remark}
	The implication $(\Leftarrow)$ holds without the conditions on the regularity of $X$ and $B$.
\end{remark}
In particular an immediate consequence of this theorem is the following.
\begin{corollary}
	In the setting of Theorem \ref{ifandonlyif}, if $B$ does contain no rational curves, then $X$ contains rational curves if and only if there exists no finite \'etale cover of the form $\tilde{B}\times E\rightarrow X$ over $B$.
\end{corollary}
The cases that are still open of Question \ref{question} (for vertical rational curves) are the following cases where $X\overset{\phi}{\rightarrow}B$ is an isotrivial fibration:
\begin{itemize}
	\item $B$ is not $\Q$-factorial so to write the canonical bundle of $X$ as the pullback of a $\Q$-Cartier $\Q$-divisor on $B$ there are some technical problems.
	\item $X$ has not log terminal singularities so we are not able to control the exceptional divisor.
	\item $X$ has log terminal singularities and $K_X\equiv_{\text{num}} \phi^* L+\sum a_i E_i$ where $E_i$ are all the exceptional divisors, $a_i$ are strictly negative number. Moreover we can't expect that these coefficients are too small or the canonical bundle can't be nef or too big by Remark \ref{logcanonicaltre}. 
\end{itemize}

%Let me give an example where we can't apply our techiques.

We stop trying to generalize Theorem \ref{generalmaintheorem} and we mention some particular cases and prove Corollary \ref{corollariofigo}.
\begin{corollary}
	Let $X$ be a projective variety of dimension $n$ with at most log terminal singularities, $\kappa(X)=n-1$ and $\tilde{q}(X)=0$. Then $X$ does contain rational curves.
\end{corollary}
\begin{proof}
	If the canonical bundle of $X$ is not nef, then there are rational curves in $X$ by the cone theorem. The numerical dimension of $K_X$ is greater than the Kodaira dimension $n-1$. If $\nu (K_X)=n $ then the canonical bundle is big and $X$ is of general type, that is a contradiction. So $\nu(K_X)=n-1=\kappa(K_X)$ and the canonical bundle is semi-ample. In particular the Iitaka fibration of the canonical bundle gives a genus one fibration $\varphi_{K_X}: X \rightarrow B$ with $K_X\sim_\mathbb{Q} \varphi_{K_X}^*(H)$ for a $\Q$-Cartier $\Q$-divisor $H$ on $B$. In particular we can apply Theorem \ref{maintheoremintro} and get the thesis.
\end{proof}

We state Theorem \ref{maintheoremintro} without boundary.
\begin{corollary}\label{withnoboundary}
	Let $X\rightarrow B $ be a relatively minimal Calabi--Yau fiber space with $dim(X)=dim(B)+1$ and $\tilde{q}(X)=0$. Then there exists a subvariety of $X$ of codimension one covered by rational curves contracted in $B$.
\end{corollary}
This result is a generalization of \cite[Theorem 0.1]{Anella}. Also for smooth varieties this theorem improves the results of \cite{diverio2016rational}, because we can certainly apply Corollary \ref{withnoboundary} to smooth varieties with finite fundamental group. 
Corollary \ref{withnoboundary} can be generalized in the direction of Theorem \ref{generalmaintheorem} and we get the following statement.

\begin{proposition}
	Let $X\xrightarrow{\phi} B $ be a genus one fibration with $\tilde{q}(X)=0$ and $K_X\equiv_{\textrm{num}} \phi^* L+\sum a_i E_i$. If we suppose that some $a_i$ is non-negative, then $X$ does contain a uniruled divisor.
\end{proposition}
\begin{proof}
	If there are no exceptional divisors we can apply Theorem \ref{maintheoremintro} to conclude. Since we are looking for a uniruled divisor and a $\Q$-factorialization is small we can assume that $X$ is $\Q$-factorial. We can write $K_X\equiv_{\textrm{num}} \phi^* L+\sum a_i E_i -\sum b_i F_i$ with $E_i\neq F_j$ the exceptional divisors and with all the coefficients non-negative. 
	The divisor $\sum  a_i E_i$ has a component $E_1$ covered by curves $C_t$ such that $\sum  a_i E_i \cdot C_t <0 $ by \cite[Lemma 2.9]{laivarietiesfibered}. Moreover the curves $C_t$ are complete intersections of the form $E_1 \cap H_1\cap \dots \cap H_{n-2}$ for some very ample divisors in $X$.
	
	Take a terminalization $\tilde{X}\xrightarrow{\nu} X$. The canonical bundle of this partial resolution is $K_{\tilde{X}}\equiv_{\textrm{num}} \nu^* K_X -\sum c_i G_i$ for some non negative number $c_i$. A component of the divisor $\nu^* E_1$ is covered by the strict transform $\tilde{C}_t$ and satisfies $\sum  a_i \nu^* E_i \cdot \tilde{C}_t <0 $. The family of curves $\tilde{C}_t $ is not contained in the other components of the support of $K_{\tilde{X}}$, so we have $K_{\tilde{X}}\cdot \tilde{C}_t< 0$. Since $\tilde{X}$ is terminal, it is smooth in codimension 2. The curves $\tilde{C}_t$ are contained in the intersection $\nu^*H_1\cap \dots \cap \nu^* H_{n-2}$. The divisors $\nu^*H_i $ are base point free, so the general element of this family does not intersect the singular points of $\tilde{X}$. This means that for a general point in $\nu^* E_1$ there is a curve contained in the regular part of $X$ that intersects negatively the canonical bundle $K_{\tilde{X}} $, hence we can apply \cite[Theorem 3.6]{Deb01} and get a family of rational curves that covers $\nu^* E_1$. Since the image of a rational curve is again rational, this implies that also $E_1$ is covered by rational curves.
\end{proof}

For smooth varieties we can do slightly better. The key point for this improvement is a very nice work on varieties covered by elliptic curves by Lazic and Peternell \cite{lazic2018maps}. For smooth varieties we can prove, using their results, the following corollary.
\begin{corollary}
	Let $X$ be a smooth projective variety of dimension $n\geq 2$ and $\tilde{q}(X)=0$. If $X$ is covered by genus one curves, then it contains a rational curve.
\end{corollary}

\begin{proof}[Proof of Theorem \ref{lazic}]
	Suppose by contradiction that $X$ does not contain rational curves. We can apply \cite[Theorem 6.12]{lazic2018maps} and find an equidimensional fiber space $X\rightarrow W$. This fibration is relatively minimal and we can proceed as in the proof of Theorem \ref{maintheoremintro} and find an irregular quasi-\'etale cover of $X$.
\end{proof}
In particular this proves the following result.

\begin{corollary}
	Let $X$ be a smooth projective variety covered by elliptic curves but with no rational curves. Then the fundamental group of $X$ is infinite.
\end{corollary}

We conclude with an useful criterion to find elliptic fibration due to K\'ollar.
\begin{theorem}\label{conditions}
	Let $X$ be a variety with at most log terminal singularities of dimension $n$, nef canonical bundle and $L$ a Cartier divisor on $X$. 
	Assume moreover 
	\begin{itemize}
		\item[1)] $L^{n-2}\cdot \operatorname{Td}_2(X)>0$.
		\item[2)] $L$ is nef.
		\item[3)] $L-\varepsilon K_X$ is nef for $0\leq\varepsilon<<1$.
		\item[4)] $L^n=0$.
		\item[5)] $L^{n-1}\neq 0 $ in $H^{2n-2}(X, \Q)$.
		
	\end{itemize}
	 Then $X$ with the Iitaka fibration associated to $L$ is a relatively minimal genus one fibration.
\end{theorem}
This result is \cite[Theorem 10]{kollar2015deformations}. In the same article there is also a log version of this theorem. 
\begin{remark}
	In \cite[Theorem 10]{kollar2015deformations} there is the further hypothesis $L^{n-1}\cdot K_X=0$, but this condition actually follows from the others. Indeed if $L-\varepsilon K_X$ is nef then $$0\leq (L-\varepsilon K_X)^n=L^n-n\varepsilon L^{n-1}\cdot K_X+...=-n\varepsilon L^{n-1}\cdot K_X+...$$
	The divisors $L$ and $K_X$ are nef, hence $L^{n-1}\cdot K_X\geq 0$. It follows that $L^{n-1}\cdot K_X=0$. 
\end{remark}
It follows from Theorem \ref{conditions} the following result.

\begin{corollary}
	Let $X$ be a variety with log terminal singularities and $\tilde{q}(X)=0$. If there exists a line bundle $L$ on $X$ such that the conditions from $1$ to $5$ of Theorem \ref{conditions} are satisfied, then $X$ does contain rational curves.
\end{corollary}

Let us conclude with an example of a regular threefold with a genus one fibration with no rational curves.

\begin{example}
	Consider in $\mathbb{P}^3$ a generic hypersurface $S$ of degree at least 5. By \cite{clemens1986curves} there are no rational curves in $S$. There is a $2:1$ cover $\tilde{S}$ of $S$ ramified along the intersection between $S$ and a quadric that is constructed as a complete intersection in $\mathbb{P}^4$. This cover comes with an involution $\eta \in \tilde{S}$ such that $S$ is the quotient of $\tilde{S}$ by this involution.
	Consider the order two automorphism $\alpha$ of the product  $Y:=\tilde{S} \times (\mathbb{C}/(\mathbb{Z}\cdot 1\oplus \mathbb{Z}\cdot i))$ defined by $\alpha(x,z)=(\eta(x), \frac{1+i}{2}-z)$. The quotient of $Y$ under this action gives a genus one fibration $X:=Y /\alpha\xrightarrow{\pi} S$. The quotient map $Y\rightarrow X $ is globally \'etale, that implies that all the fibers of $\pi$ are genus one curves. The holomorphic one forms on $X$ are exactly the holomorphic one forms on $Y$ that are invariant under the action of $\alpha$. Since $\tilde{S}$ is a complete intersection in $\mathbb{P}^4$ it is regular. A one forms $\omega$ on $Y$ can be written as $\pi^* \beta  $ where $\pi$ is the projection from $Y$ to $ E:=\mathbb{C}/(\mathbb{Z}\cdot 1\oplus \mathbb{Z}\cdot i)$. By the choice of the automorphism on $E$ we see that $\alpha^* \omega=-\omega$, hence $q(X)=0$. Since all the vertical curves have genus one and an horizontal rational curve gives a rational curve on $S$, there are no rational curves on $X$. So this is an example of a regular genus one fibration with no rational curves.
\end{example}

%\begin{example}
%	Let $S\subset \mathbb{P}^3 $ a very general surface of high degree. It is well known that this surface is hyperbolic, in particular $S$ does not contain rational curves. The line bundle $\mathcal{O}_S(2)$ gives us a ramified two to one cover $\tilde{S}$ of $S$. This cover is the quotient under an involution $i$ of $\tilde{S}$. Consider the order two automorphism $\alpha$ of $Y:=\tilde{S}\times \mathbb{C}/(\mathbb{Z}\cdot 1\oplus \mathbb{Z}\cdot i)$ defined by $\alpha(s,z)=(i(s), \frac{1+i}{2}-z)$. The quotient $X:=Y/\alpha\xrightarrow{\pi} S$ is a genus one fibration. The quotient map $Y\rightarrow X $ is globally \'etale, that implies that all the fiber of $\pi$ are genus one curves. So there exists no vertical rational curves on $X$ and the orizontal rational curves give rational curves on $S$, so $X$ does not contain rational curves. Moreover $\tilde{q}(X)=1$. The global sections of $\Omega^1_X$ are the global sections of $\Omega_Y^1$ that are invariant under the action of $\alpha$. Since $H^0(Y, \Omega_Y^1)=p^* H^0(E, \Omega_E)$ and by the definition of $\alpha$, one can show that $X$ is regular. So this is an example that shows that working only with the irregularity is not sufficient to find rational curves on a genus one fibration for a smooth projective variety. (perché S' è regolare?)
%\end{example}

\linespread{0.9}
\footnotesize

	%\addcontentsline{toc}{chapter}{bibtesi}
	
%\printbibliography

\bibliographystyle{alpha}
\linespread{0.9}
\footnotesize 
\bibliography{C:/Users/fabrizio/Desktop/uni/personali/biblio}{}

%\author{Fabrizio Anella}

\end{document}